\documentclass[12pt, reqno]{amsart}
\usepackage{amsmath}
\usepackage{dsfont}
\usepackage{amsfonts}
\usepackage{amssymb}
\usepackage{aliascnt}
\usepackage{graphicx}
\usepackage{mathrsfs}
\usepackage{enumerate}
\usepackage[bookmarks=true,pdfstartview=FitH, pdfborder={0 0 0}, colorlinks=true,citecolor=blue, linkcolor=blue]{hyperref}

\newtheorem{theorem}{Theorem}[section]
\newaliascnt{conj}{theorem}
\newaliascnt{cor}{theorem}
\newaliascnt{lemma}{theorem}
\newaliascnt{fact}{theorem}
\newaliascnt{claim}{theorem}
\newaliascnt{prop}{theorem}
\newaliascnt{definition}{theorem}

\newtheorem{cor}[cor]{Corollary}
\newtheorem{lemma}[lemma]{Lemma}

\newtheorem{prop}[prop]{Proposition}

\newtheorem{definition}[definition]{Definition}

\aliascntresetthe{conj} \aliascntresetthe{cor}
\aliascntresetthe{lemma} \aliascntresetthe{fact}
\aliascntresetthe{claim} \aliascntresetthe{prop}
\aliascntresetthe{definition}

\theoremstyle{definition}
\newaliascnt{example}{theorem}

\aliascntresetthe{example}

\theoremstyle{remark}
\newaliascnt{rmk}{theorem}
\newtheorem{remark}[rmk]{Remark}
\aliascntresetthe{rmk} 

\def\sek~{\S{}}

\numberwithin{equation}{section}
\newcommand{\inj}{\operatorname{inj}}

\newcommand{\sectional}{\operatorname{sec}}

\newcommand{\dist}{\operatorname{d}}

\renewcommand{\span}{\operatorname{span}}
\newcommand{\RR}{\mathbb{R}}

\renewcommand{\SS}{\mathbf{S}}

\begin{document}
\title{$1/4$-Pinched Contact Sphere Theorem}
\author{Jian Ge}
\author{Yang Huang}
\maketitle
\begin{abstract}
Given a closed contact 3-manifold with a compatible Riemannian metric, we show that if the sectional curvature is 1/4-pinched, then the contact structure is universally tight. This result improves the Contact Sphere Theorem in~\cite{EKM2012}, where a 4/9-pinching constant was imposed. Some tightness results on positively curved contact open 3-manifold are also discussed.
\end{abstract}
\section{Introduction}
A {\em contact metric manifold} $(M, \xi, g)$ is a contact manifold equipped with a compatible Riemannian metric $g$, where $\xi$ is the contact structure. See \autoref{def:comp-metric} for the definition of compatibility. In~\cite{EKM2012}, the authors studied how the curvature bounds on $g$ implies the (universal) tightness of $\xi$. In particular, the authors showed that if the sectional curvature of $g$ is $4/9$-pinched, then $\xi$ is universally tight. Since any 1/4-pinched closed 3-manifold has the universal cover diffeomorphic to $S^3$ and there is a unique, up to contactomorphism, tight contact structure $\xi_{std}$ on $S^3$~\cite{Eli1992}, the authors therefore concluded that the universal cover of $(M,\xi)$ must be contactomorphic to $(S^3,\xi_{std})$. The main goal of this note is to improve the pinching constant to 1/4. More precisely, we have
\begin{theorem}[Contact Sphere Theorem]\label{thm:Sphere}
Suppose $(M,\xi,g)$ is a closed contact metric 3-manifold. If the sectional curvature $\sectional(g)$ satisfies
$$
\frac14 < \sectional(g)\le 1,
$$
then the universal cover of $M$, with the lifted contact structure, is contactomorphic to $(S^3,\xi_{std})$.
\end{theorem}

\begin{remark}
According to~\cite{Ham1982}, a closed Riemannian 3-manifold with the sectional curvature pinched by any positive number has the universal cover diffeomorphic to $S^3$ (In fact, Hamilton shows the positivity of Ricci curvature is preserved along Ricci flow and converges to constant sectional curvature). At this moment we do not know whether the pinching constant 1/4 is optimal or not.
\end{remark}

Our strategy of proving Theorem~\ref{thm:Sphere} essentially follows the arguments in~\cite{EKM2012}. However, instead of trying to bound the tight radius by convex radius from below as in~\cite{EKM2012}, we construct a shrinking family of (not necessarily smooth) strictly convex spheres and carefully estimate the convexity to ensure the tightness.

Using this new tool, we are able to prove the following two theorems for open manifolds, which can be viewed as a counterpart of Corollary 1.4 in \cite{EKM2012}. Note also that there is no convexity radius estimate on such manifold since there is no upper curvature bound is assumed.

\begin{theorem}\label{thm:R3}
Let $(M, \xi,g)$ be an open contact metric manifold such that $g$ is complete and $\sec(g)>0$, then $\xi$ is tight.
\end{theorem}

We note immediately that $M$ being open and being positively curved imply that $M \simeq \RR^3$. In fact using an argument of Wu~\cite{Wu1979}, we can weaken the curvature condition and get

\begin{theorem}\label{thm:wR3}
Let $(M, \xi,g)$ be an open contact metric manifold such that $g$ having nonnegative sectional curvature on $M$ and positive sectional curvature in $M\setminus K$, where $K$ is a compact subset of $M$, then $\xi$ is tight.
\end{theorem}

This paper is organized as follows. In \autoref{s:1} we review the notion of $\epsilon$-convexity in Riemannian geometry, and in particular, we establish a convexity estimate which is the key ingredient in our proof \autoref{thm:Sphere}. In \autoref{s:2} we compare the the Riemannian convexity and pseudo-convexity in almost complex manifold. Proofs of \autoref{thm:Sphere}, \autoref{thm:R3} and \autoref{thm:wR3} are given in \autoref{s:3} and \autoref{s:4}.\\

It is our pleasure to thank Werner Ballmann for useful discussions, and John Etnyre for helpful comments.

\section{Riemannian Convexity}\label{s:1}
In this section, we will study some convexity properties of a $3$-dimensional Riemannian manifold $(M, g)$. We restrict ourself to dimension three because we are interested in the geometry of contact 3-manifold, but all the results in this section still hold in any dimension. We say a domain $D\subset M$ is {\em convex}, if any two points $x, y\in D$ can be joint by a minimal geodesic $\gamma$ contained in $D$. In this note, we denote the distance function induced by the Riemannian metric by $\dist(\cdot, \cdot)$ and an open ball of radius $r$ at $p$ by $B(p, r)$. By writing $\partial B(p, r)$ we mean the distance sphere of radius $r$ centered at $p$.

\begin{lemma}\label{lem:CplmtisConv}
Let $(M, g)$ be a closed Riemannian manifold with $$\sectional(g)\ge 1,$$ then for any $p\in M$, the close set $N:=M\setminus B(p, \pi/2)$ is convex.
\end{lemma}
\begin{proof}
Without loss of generality, we may assume $N$ contains at least two distinct points, say, $x, y\in N$. Let $\gamma: [0, a] \to M$ be a geodesic parametrized by arc-length with $\gamma(0)=x, \gamma(a)=y$ such that $a=\dist(x, y)$. Since $\sectional(g)\ge 1$, the diameter of $M$ is less or equal to $\pi$ with equality holds if and only if $M$ is isometric to the round sphere. Hence we can further assume $a<\pi$, otherwise if $a=\pi$ then $N$ is a round hemisphere and the statement is clear. Choose the comparison triangle $\Delta\tilde{p}\tilde{x}\tilde{y}$ in $\SS^2(1)$, the unit $2$-sphere, i.e., the points in $\SS^2(1)$ such that $\dist_{S^2}(\tilde{p},\tilde{x})=\dist(p,x)$, $\dist_{S^2}(\tilde{p},\tilde{y})=\dist(p,y)$ and $\dist_{S^2}(\tilde{x},\tilde{y}) =\dist(x,y)$, where $\dist_{S^2}$ is induced by the round metric on $\SS^2(1)$. Let $\tilde{\gamma}$ be the minimal geodesic in $\SS^2(1)$ connecting $\tilde{x}$ and $\tilde{y}$. Then Toponogov's comparison theorem implies that
$$
\dist(p, \gamma(t)) \ge \dist(\tilde{p}, \tilde{\gamma}(t))
$$
for all $t$.
It is clear that $\dist(\tilde{p}, \tilde{\gamma}(t))\ge \pi/2$ in $\SS^2(1)$. Hence $\gamma(t)\in N$ for any $t\in [0, a]$, i.e., $N$ is convex.
\end{proof}

For example, if we take $M$ to be the round 3-sphere $\SS^3(1)$, then the $N$ defined above is a hemisphere with a totally geodesic boundary, which is a great 2-sphere. However if we shrink $N$ a little bit, then we get a smaller hemisphere with strictly convex boundary. It is helpful to keep this example in mind because a similar argument will be used later to construct a convex ball in a contact metric 3-manifold.

For our later purposes, we need to study continuous convex functions. One quick way to define the convexity of a continuous function $f: M \to \RR$ is to require that its restriction on any geodesic segment is convex as a function from $\RR$ to $\RR$. But this definition is not good enough for our purposes, so we need the following qualitative definition from \cite{Esc1996}.

\begin{definition}\label{def:e-convexity}
Let $(M, g)$ be a Riemannian manifold and $\epsilon>0$ be a constant. A continuous function $f: M\to \RR$ is called {\em $\epsilon$-convex} if for any point $p\in M$ and any $0<\eta<\epsilon$, there exists a smooth function $h$, defined in an open neighborhood $U$ of $p$, such that
\begin{itemize}
	\item{$h\le f$ in $U$,}
	\item{$h(p)=f(p)$,}
	\item{$D^2h(v, v) \ge \eta \| v\|^2  \text{ for any } v\in T_p(M).$}
\end{itemize}
Such an $h$ is called a {\em $\eta$-supporting function} of $f$ at $p$, or just supporting function when $\eta$ is implicit.
\end{definition}

The following estimation of the convexity of distance function to the boundary is crucial for our proof. Some similar statements for Busemann function have been proved by Cheeger-Gromoll~\cite{CG1972} and Wu~\cite{Wu1979}. For distance function to boundary of Alexandrov spaces with lower curvature bound, it is first proved by Perelman in~\cite{Per1991} and than made rigors by Alexander-Bishop~\cite{AB2003}. However the proof given in~\cite{AB2003} used several important tools and fundamental structural theory for Alexandrov spaces and hence require more background knowledge. In order to keep this note more self-contained we present a pure Riemannian geometric proof, which is more accessible for general readers. Which also has the advantage that the explicit supporting function is constructed, which can be used later.

\begin{prop}[Convexity Estimate]\label{prop:Supporting}
Let $(M, g)$ be a Riemannian manifold with $\sectional(g)\ge 1$, and $D\subset M$ be a closed convex domain with nonempty boundary. Set
$$
f(x)=e^{\pi/2-\dist(x, \partial D)},
$$
then $f(x)$ is strictly convex. Moreover if $\dist(x, \partial D)=\ell$ then $f$ is $\min(1, \ell)$-convex.
\end{prop}
\begin{proof}
We first setup some notations which will be used in the proof. Let $p\in D\setminus \partial D$, $q\in \partial D$ such that $$\ell=\dist(p, q)=\dist(p, \partial D).$$ Let $\gamma: (-a, a)\to D$ be a unit speed geodesic with $\gamma(0)=p$, and $\sigma: [0, \ell]\to D$ be the geodesic from $p$ to $q$. Hence at $p$, we have the following orthogonal decomposition
$$
\gamma'(0)=a\sigma'(0)+bW,
$$
where $W\in (\sigma'(0))^{\perp}\subset T_p(M)$ and $a^2+b^2=1$. Using a method similar to the definition of variational field in~\cite{Ge2013}, we construct a vector field $V$ along $\sigma$ by
\begin{equation}\label{eq:prop:V(t)}
V(t)=a\left(1-\frac{t}{\ell}\right)\sigma'(t)+bW(t),
\end{equation}
where $W(t)$ is the parallel translate of $W$ along $\sigma$. Let $\alpha:[0, \ell]\times (-\delta, \delta)\to M$ be a variation of $\sigma$ which induces $V(t)$ for $\delta<a$, in other words, $\alpha(t, s)=\exp_{\sigma(t)}(sV(t))$ for $t\in[0,\ell],s\in(-\delta,\delta)$. Since $D$ is convex, and $W(\ell)\perp \sigma'(\ell)$, we have
\begin{equation}\label{eq:prop:02}
\alpha(\ell, s)\not\in D\setminus \partial D \text{ for all } s.
\end{equation}

Now the proof proceeds in two steps.\\

\noindent
\textbf{Step 1}: Constructing a supporting function.\\

Denote the curve $t\to \alpha(t, s)$ by $\sigma_s$. We define $L:(-\delta, \delta)\to \RR$ by
$$
L(s)=\frac{\pi}{2}-\int_0^\ell \sqrt{\langle \sigma_s'(t), \sigma_s'(t) \rangle} dt.
$$
i.e., $L$ is the negative of the length of $\sigma_s$. Clearly $L(0)=\pi/2-\ell$. Let
$$
h(s)=\frac{\pi}{2}-\dist(\gamma(s), \partial D),
$$
Then by \eqref{eq:prop:02}
$$
L(s)\le \frac{\pi}{2}-\dist(\gamma(s), \partial D) = h(s).
$$
i.e., $L$ is a supporting function of $h$ at $\gamma(0)$.\\

\noindent
\textbf{Step 2}: Convexity estimates of $L$ and $e^L$.\\

By the second variational formula for arc-length, we have
$$
L''(0)=\int_0^\ell \Big( R(\sigma', V', \sigma', V')- \langle V', V' \rangle + (\langle V', \sigma' \rangle)^2 \Big) dt.
$$
Here $R(X, Y, Z, W)=\langle -\nabla_X\nabla_Y Z +\nabla_Y\nabla_X Z + \nabla_{[X, Y]}Z, W\rangle$ is the Riemannian curvature tensor. Note that our construction \eqref{eq:prop:V(t)} implies $\langle V', V' \rangle =a^2/\ell^2$ and $(\langle V', \sigma' \rangle)^2= a^2/\ell^2$. Moreover by the assumption that $\sectional (g) \ge 1$, we have
$$
L''(0)\ge \int_0^\ell b^2 dt = b^2\ell.
$$
Consider the composition $e^L$, we calculate as follows
\begin{equation}
\begin{aligned}
(e^L)''(0) &= e^L (L')^2(0)+e^L L'' (0) \\
&=e^{\frac{\pi}{2}-\ell}(a^2+b^2\ell)\\
&\ge \min(\ell, 1)
\end{aligned}
\end{equation}
where we used the first variational formula to get $L'(0)=a$. The last inequality follows from the fact that $\ell \le \pi/2$ under given curvature condition, and $a^2+b^2=1$. By Step 1, $e^L$ supports $e^h$, therefore $f=e^h$ is $\min(\dist(p, \partial D), 1)$-convex.
\end{proof}

\begin{remark}
Our choice of using the exponential function in the construction of $f$ is not essential, and in fact, any function $\kappa:\RR\to \RR$ with $\kappa>0$, $\kappa'>0$ and $\kappa''>0$ will also work. As a consequence the number $\min(1, \ell)$ is also not important since it clearly depends on the choice of $\kappa$. In fact most commonly used functions in metric geometry are \emph{generalized trigonometric functions}, which interpolate analytically between the usual trigonometric and hyperbolic functions. See for example~\cite{AB2003} or~\cite{Ge2013}.
\end{remark}
\section{Pseudoconvexity in symplectizations}\label{s:2}
Let $(M, \xi)$ be a contact $3$-manifold. Following~\cite{EKM2012} we have

\begin{definition}\label{def:comp-metric}
A Riemannian metric $g$ is {\em compatible} with $\xi$ if there is a contact form $\alpha$ defining $\xi$ such that $$||R_\alpha||=1 \text{ and } \ast d\alpha=\theta'\alpha$$ for some positive constant $\theta'$, where $R_\alpha$ is the Reeb vector field and $\ast$ is the Hodge star operator associated with $g$. A compatible triple $(M,\alpha,g)$ is called a {\em contact metric manifold}.
\end{definition}

\begin{remark}
A compatibility condition between contact structure and Riemannian metric first appeared in \cite{CH1985}, where $\theta'=2$.
\end{remark}

\begin{remark}
In \cite{EKM2012}, a notion of {\em weakly compatible metric} is also discussed, where $\theta'$ is not necessarily a constant and the length of $R_\alpha$ is allowed to vary. But we will not use the weak compatibility in this note.
\end{remark}

Throughout this section we will assume that $(M,\alpha,g)$ is a contact metric manifold with $\sectional(g)\ge 1$.

Recall the symplectization $W= \RR_+ \times M$ of $M$ is a symplectic manifold with symplectic form $\omega=d(t\alpha)$ where $t\in\RR_+$. Also fix a metric-preserving compatible almost complex structure $J$ on $W$ such that
$$
J\partial_t = R_\alpha \text{ and } J\xi=\xi.
$$
Let $D \subset M$ be a closed convex domain with boundary and
$$
f(x)=e^{\frac{\pi}{2}-\dist(x, \partial D)}
$$
be a strictly convex function on $D\setminus \partial D$ according to \autoref{prop:Supporting}. We extend $f$ to a function $\tilde{f}$ on $\RR_+ \times N$ by $\tilde{f}(t, x)= f(x)$ for $x\in N, t\in \RR_+$. The following proposition is crucial for our later detection of overtwistedness.

\begin{prop}[Weak Maximal Principle]\label{prop:nonTangent}
Using the notations from above, for $e^{\pi/2}>c>\min{f}$, let $\Omega^c:= \RR_+ \times f^{-1}((-\infty, c])$ and $\Sigma^c:= \RR_+ \times f^{-1}(c)=\partial \Omega^c$ Then the interior of any $J$-holomorphic curve $C$ in $\Omega^c$ is disjoint from $\Sigma^c$.
\end{prop}
\begin{proof}
Since $c<e^{\pi/2}$, \autoref{prop:Supporting} implies that $f$ is $\tau$-convex for some $\tau>0$ depends only on $c$. Suppose $\Sigma^c\cap int(C)$ is nonempty, then it contains a point, say, $p$. By the proof of \autoref{prop:Supporting}, there exists a supporting function $g$ of $f$ at $p$. Denote by $\Sigma'$ the hypersurface $g^{-1}(c)\times \RR_+$ in $W$. The following calculation of the Levi form $L_{\tilde g}$ of $\tilde{g}$ is obtained in \cite{EKM2012} (Proposition 3.7), where $\tilde g$ is the usual extension of $g$ on $W$. For any unit vector $v \in T\Sigma' \cap JT\Sigma'$, we can write $v=a \partial_t + b R_\alpha + e v_0$, where $v_0\perp\span {(\partial_t, R_\alpha)}$ a unite vector, hence $a^2+b^2+e^2=1$.
$$
L_{\tilde g}(v, v)=D^2\tilde{g}(v,v) + D^2\tilde{g}(Jv, Jv) \ge \tau (2-a^2-b^2)\ge \tau>0
$$
since $\tilde{g}$ is constant in the $\RR_+$-direction and $h$ is $\tau$-convex by construction. In particular $\Sigma'$ is strictly pseudoconvex. On the other hand, it is easy to see that $C$ is tangent to the smooth hypersurface $\Sigma'$ at $p$, which contradicts the pseudoconvexity of $\Sigma'$.
\end{proof}

\begin{remark}
Note that all the level surfaces $\Sigma^c$, $e^{\pi/2}>c>\min{f}$, are topologically spheres, thanks to the strict convexity of $f$.
\end{remark}

An important consequence is
\begin{cor}\label{cor:Convex=>Tight}
For any $e^{\pi/2}>c>\min{f}$, $f^{-1}((-\infty, c])$ is a tight ball.
\end{cor}
\begin{proof}
It is easy to see that $f^{-1}((-\infty, c])$ is homeomorphic to a ball. Arguing by contradiction, suppose there exists an overtwisted disk $D_{OT}\subset f^{-1}((-\infty, c])$, then \autoref{prop:nonTangent} guarantees that Hofer's proof of the Weinstein conjecture for overtwisted contact structures \cite{Hof1993} carries over in our situation to produce a closed Reeb orbit $\gamma$. Consider the (trivial) $J$-holomorphic cylinder $C_\gamma:=\RR_+ \times \gamma$ contained in the interior of $\Omega^c$. Define
$$
T=\inf \{0>t>c ~|~ C_\gamma\subset \Omega^t, C_\gamma \cap \Sigma^t  =\emptyset \}.
$$
Then clearly $\Omega^T$ intersects the interior of $C_\gamma$ nontrivially, which contradicts \autoref{prop:nonTangent}.
\end{proof}

\section{Proof of \autoref{thm:Sphere}}\label{s:3}
In this section we assume that $(M, \xi,g)$ satisfies the assumptions in \autoref{thm:Sphere}. Passing to the universal cover if necessary, we may further assume that $M$ is simply connected. Note that the compactness condition is preserved due to the positivity of $\sec(g)$. We start by a slight refinement of \autoref{lem:CplmtisConv} as follows.
\begin{lemma}\label{cor:B_conv}
Given $\sectional(g)>1/4$, there exists $\delta>0$ such that the set $M\setminus B(p, (1-\delta)\pi)$ is convex.
\end{lemma}
\begin{proof}
Since $M$ is compact, there exists $\epsilon>0$ such that $\sectional(g) \ge \frac14+\epsilon$
Rescaling the metric to have lower curvature bound $1$ and applying \autoref{lem:CplmtisConv} gives the desired convexity. In fact one can take $\delta>0$ such that
$$
(1-\delta)\pi=\frac{\pi}{2\sqrt{\frac14+\epsilon}}.
$$
\end{proof}

Now let's assume the curvature is $1/4$-pinched and let $N=B(p, (1-\delta)\pi)$. Consider the distance function $h(x) =\dist(x, \partial N )$. Let
$$
B_1=B(p,\pi) \quad\text{and}\quad
B_2= \{x\in N \ |\ h(x)\ge \frac{\delta}{2}\pi\}
$$
Then by \autoref{prop:Supporting} and \autoref{cor:Convex=>Tight}, $B_2$ is a tight ball. Now $M$ is coved by two balls:
$$
M=B_1 \cup B_2.
$$
Moreover we note that $\partial B_2$ is contained in the interior of $B_1$ and also $\partial B_1$ is contained in $B_2$. See \autoref{2dimttdec}.

\begin{figure*}[ht]
\includegraphics[width=200pt]{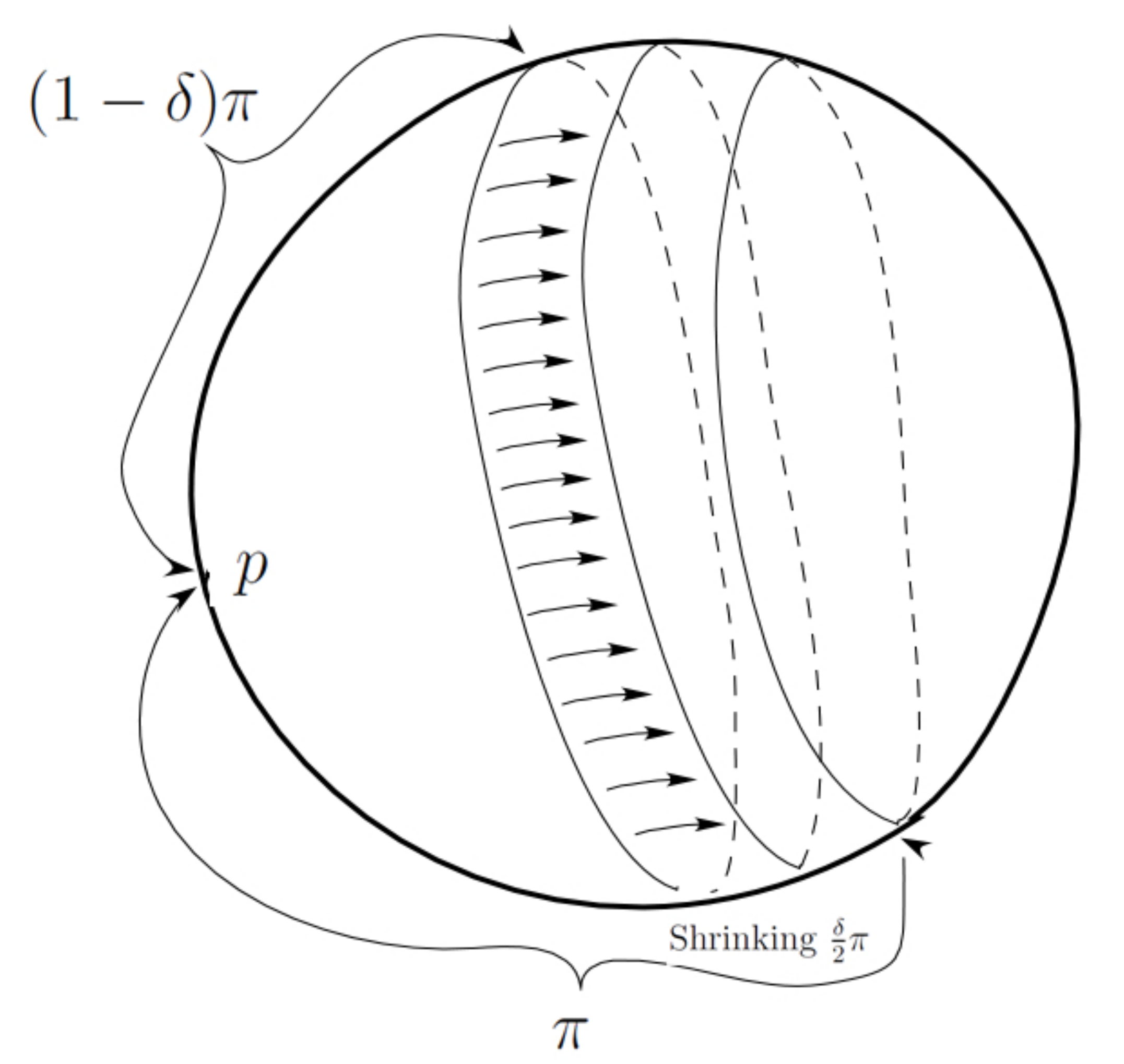}\\
\caption{Covering $M$ by two balls }\label{2dimttdec}
\end{figure*}

Before complete the proof of \autoref{thm:Sphere}, we need the following theorem from \cite{EKM2012}, which tells us where to look for overtwisted disks.

\begin{theorem}[\cite{EKM2012} Theorem 1.2]\label{thm:EKM}
Let $(M,\xi,g)$ be a contact metric $3$-manifold and $\inj(g)$ be the injective radius of $g$. Fix a point $p\in M$. Suppose $B(p,r)$ is overtwisted for some $r<\inj(g)$. Then for any $r\le R <\inj(g)$, the geodesic sphere $S(p, R)$ contains an overtwisted disk.
\end{theorem}

Now we are ready to finish the proof of our contact sphere theorem.
\begin{proof}[Proof of \autoref{thm:Sphere}]
Recall $M=B_1 \cup B_2$ and $B_2$ is tight. Arguing by contradiction, suppose there exists an overtwisted disk $D_{OT} \subset M$. By Eliashberg's classification of tight contact structures on 3-ball \cite{Eli1992}, there exists a radial contact vector field on $B_2$ whose flow induces a contact isotopy $\phi_t: M \to M$, $t\in[0,1]$, such that $$\phi_0=\text {id},~ \phi_t |_{M \setminus B_2} =\text{id}, \text{ and } \phi_1(D_{OT}) \subset B_1.$$ The last assertion follows from the fact that a small neighborhood of $\partial B_2$ is contained in the interior of $B_1$. Now \autoref{thm:EKM}, applied to $B_1$, implies that for sufficiently small $\epsilon>0$ and $\pi-\epsilon < r_0 < \pi$, the geodesic sphere $S(p, r_0)$ contains an overtwisted disk. But $S(p, r_0) \subset B_1$ for small $\epsilon$, which contradicts the fact the $B_1$ is tight.

Now the classical $1/4$-pinched sphere theorem implies that $M$ is homeomorphic to $S^3$, and since we are in dimension 3, it is diffeomorphic to $S^3$. Again, Eliashberg's uniqueness theorem of tight contact structures on $S^3$ implies that $(M,\xi)$ must be contactomorphic to $(S^3,\xi_{std})$.
\end{proof}

\section{Proof of \autoref{thm:R3} and \autoref{thm:wR3}}\label{s:4}
In this section we prove \autoref{thm:R3} and sketch a proof of \autoref{thm:wR3}. Although \autoref{thm:wR3} implies \autoref{thm:R3}, we would like to emphasize the proof of the positively curved case, which shows how the strictly convexity of the Busemann function played a role. The idea is to construct a strictly convex exhaustion function. Define the Busemann function
$$
b(x)=\lim_{t\to\infty} t- \dist(x, \partial B(p, t)).
$$
where $p\in M$ is a fixed point. It is showed in \cite{CG1972} (cf. also \cite{Wu1979}) that $b$ satisfies the following properties
\begin{enumerate}
\item $b$ is a strictly convex Lipschitz function bounded from below by $a_0>-\infty$;
\item $b$ is a exhaustion function, i.e., if we denote $C^t:=b^{-1}((-\infty, t])$, then for all $c\ge a_0$, $C^t$ is compact and $M=\cup_{t\ge a_0}C^t$;
\item for $a_0\le t<s$, if $x\in \partial C^t$, then $\dist(x, \partial C^s)= s-t$, in other words, $b$ is the distance to the boundary of $C^s$ up to a constant.
\end{enumerate}
The third property shows that Busemann function can be viewed as certain distance function from infinity and this is exactly why our convexity estimate also works for $b$. In fact the compactness of $C^t$ implies the sectional curvature in $C^t$ is bounded:
$$
0<\delta\le \sectional(g)|_{C^t} \le \Delta,
$$
for some positive constants $\delta$ and $\Delta$ depend on $t$. Hence we can rescale the metric and apply \autoref{prop:Supporting} and \autoref{cor:Convex=>Tight} to $C^t$. (In fact we need to rescale the metric such that it has lower bound 1.) Therefore, $C^t$ is tight. Since $M=\cup_{t\ge a_0} C^t$, $M$ itself is tight. This finishes the proof of \autoref{thm:R3}.

Finally, we sketch the proof of \autoref{thm:wR3}, which is along the same line as in the previous proof. The only difficulty is that the sectional curvature is only nonnegative, hence the \autoref{prop:Supporting} does not apply. However Theorem C(a) in \cite{Wu1979} shows that $b$ is an essentially convex function, i.e. there exists $\epsilon>0$ such that $\kappa\circ b$ is $\epsilon$-convex for smooth $\kappa:\RR\to \RR$ with $\kappa>0 , \kappa'>0$ and $\kappa''>0$. Hence by \autoref{cor:Convex=>Tight}, $C^t$ is tight, hence $M$ is tight.

\bibliography{CTS}

\providecommand{\bysame}{\leavevmode\hbox to3em{\hrulefill}\thinspace}
\providecommand{\MR}{\relax\ifhmode\unskip\space\fi MR }
\providecommand{\MRhref}[2]{%
  \href{http://www.ams.org/mathscinet-getitem?mr=#1}{#2}
}
\providecommand{\href}[2]{#2}
\begin{thebibliography}{EKM12}

\bibitem[AB03]{AB2003}
Stephanie Alexander and Richard~L. Bishop, \emph{{FK}-convex functions on
  metric spaces}, Manuscripta Math. \textbf{110} (2003), no.~1, 115--133.
  \MR{1951803 (2004a:53100)}

\bibitem[CG72]{CG1972}
Jeff Cheeger and Detlef Gromoll, \emph{On the structure of complete manifolds
  of nonnegative curvature}, Ann. of Math. (2) \textbf{96} (1972), 413--443.
  \MR{0309010 (46 \#8121)}

\bibitem[CH85]{CH1985}
S.~S. Chern and R.~S. Hamilton, \emph{On {R}iemannian metrics adapted to
  three-dimensional contact manifolds}, Workshop {B}onn 1984 ({B}onn, 1984),
  Lecture Notes in Math., vol. 1111, Springer, Berlin, 1985, With an appendix
  by Alan Weinstein, pp.~279--308. \MR{797427 (87b:53060)}

\bibitem[EKM12]{EKM2012}
John~B. Etnyre, Rafal Komendarczyk, and Patrick Massot, \emph{Tightness in
  contact metric 3-manifolds}, Invent. Math. \textbf{188} (2012), no.~3,
  621--657. \MR{2917179}

\bibitem[Eli92]{Eli1992}
Yakov Eliashberg, \emph{Contact {$3$}-manifolds twenty years since {J}.
  {M}artinet's work}, Ann. Inst. Fourier (Grenoble) \textbf{42} (1992),
  no.~1-2, 165--192. \MR{1162559 (93k:57029)}

\bibitem[Esc86]{Esc1996}
J.-H. Eschenburg, \emph{Local convexity and nonnegative curvature---{G}romov's
  proof of the sphere theorem}, Invent. Math. \textbf{84} (1986), no.~3,
  507--522. \MR{837525 (87j:53080)}

\bibitem[Ge13]{Ge2013}
Jian Ge, \emph{Comparison theorems for manifold with mean convex boundary},
  Preprint, 2013.

\bibitem[Ham82]{Ham1982}
Richard~S. Hamilton, \emph{Three-manifolds with positive {R}icci curvature}, J.
  Differential Geom. \textbf{17} (1982), no.~2, 255--306. \MR{664497
  (84a:53050)}

\bibitem[Hof93]{Hof1993}
H.~Hofer, \emph{Pseudoholomorphic curves in symplectizations with applications
  to the {W}einstein conjecture in dimension three}, Invent. Math. \textbf{114}
  (1993), no.~3, 515--563. \MR{1244912 (94j:58064)}

\bibitem[Per93]{Per1991}
G.~Perelman, \emph{{A. D.} {A}lexandrov's spaces with curvatures bounded from
  below. {II}}, Preprint, 1993.

\bibitem[Wu79]{Wu1979}
H.~Wu, \emph{An elementary method in the study of nonnegative curvature}, Acta
  Math. \textbf{142} (1979), no.~1-2, 57--78. \MR{512212 (80c:53054)}

\end{thebibliography}
\bibliographystyle{amsalpha}
\bigskip
\noindent
{\small Jian Ge and Yang Huang}\\
{\small Max Planck Institute for Mathematics } \\
{\small Vivatsgasse 7}\\
{\small 53111 Bonn, Germany} \\
{\small jge@mpim-bonn.mpg.de and yhuang@mpim-bonn.mpg.de} \\

\end{document}